\documentclass[11pt]{amsart}
\usepackage{palatino}
\usepackage{amsfonts}
\usepackage{amssymb}
\usepackage{amscd}
\usepackage[breaklinks,bookmarksopen,bookmarksnumbered]{hyperref}
\usepackage[dvips]{graphicx}
\input xy
\xyoption{all}
\usepackage[all]{xy}

\newtheorem{theorem}{Theorem}[section]
\newtheorem{proposition}[theorem]{Proposition}
\newtheorem{corollary}[theorem]{Corollary}
\newtheorem{lemma}[theorem]{Lemma}
\theoremstyle{definition}

\newtheorem{remark}[theorem]{Remark}

\newtheorem{conjecture/question}[theorem]{Conjecture/Question}

\newtheorem{remark/definition}[theorem]{Remark/Definition}
\newtheorem{terminology/notation}[theorem]{Terminology/Notation}

\def\ZZ{{\mathbb Z}}

\def\PP{{\textbf P}}
\def\OO{\mathcal{O}}

\def\F{\mathcal{F}}
\def\P{\mathcal{P}}

\def\E{\mathcal{E}}

\def\K{\mathcal{K}}

\def\cM{\mathcal{M}}

\def\cU{\mathcal{U}}

\def\Pic0{{\rm Pic}^0(X)}

\newcommand {\Min}{\mbox{min}}
\newcommand {\gon}{\mbox{gon}}

\def\rmapdown#1{\Big\downarrow
   \rlap{$\vcenter{\hbox{$\scriptstyle#1$}}$ }}

\def\Cliff{\mbox{Cliff}}

\begin{document}
 \title[ ]{Higher rank Brill-Noether theory on sections of $K3$ surfaces}

\author[G. Farkas]{Gavril Farkas}
\address{Humboldt-Universit\"at zu Berlin, Institut F\"ur Mathematik,  Unter den Linden 6
\hfill \newline\texttt{}
 \indent 10099 Berlin, Germany} \email{{\tt farkas@math.hu-berlin.de}}

 \author[A. Ortega]{Angela Ortega}
\address{Humboldt-Universit\"at zu Berlin, Institut F\"ur Mathematik,  Unter den Linden 6
\hfill \newline\texttt{}
 \indent 10099 Berlin, Germany} \email{{\tt ortega@math.hu-berlin.de}}

\thanks{}

\begin{abstract}
We discuss the role of K3 surfaces in the context of Mercat's conjecture in higher rank Brill-Noether theory. Using liftings of Koszul classes, we show that Mercat's conjecture in rank 2 fails for any number of sections and for any gonality stratum along a Noether-Lefschetz divisor inside the locus of curves lying on K3 surfaces. Then we show that Mercat's conjecture in rank 3 fails even for curves lying on K3 surfaces with Picard number 1. Finally, we provide a detailed proof of Mercat's conjecture in rank 2 for general curves of genus 11, and describe explicitly the action of the Fourier-Mukai involution on the moduli space of curves.
\end{abstract}

\maketitle

\section{Introduction}
The Clifford index $\mathrm{Cliff}(C)$ of an algebraic curve $C$ is the second most important invariant of $C$ after the genus, measuring the complexity of the curve in its moduli space. Its geometric significance is amply illustrated for instance in the statement
$$K_{p, 2}(C, K_C)=0\Leftrightarrow  p<\mathrm{Cliff}(C)$$ of Green's Conjecture \cite{G} on syzygies of canonical curves.
It has been a long-standing problem to find an adequate generalization of $\mathrm{Cliff}(C)$ for higher rank vector bundles. A definition in this sense has been proposed by Lange and Newstead \cite{LN1}: If $E\in \cU_C(n, d)$ denotes a semistable vector bundle of rank $n$ and degree $d$ on a curve $C$ of genus $g$, one defines its Clifford index as
$$\gamma(E):=\mu(E)-\frac{2}{n} h^0(C, E)+2\ge 0,$$ and then the \emph{higher Clifford indices} of $C$ are defined as the quantities
$$\mathrm{Cliff}_n(C):=\mbox{min}\bigl\{\gamma(E): E\in \cU_C(n, d), \ \ d\leq n(g-1), \ \ h^0(C, E)\geq 2n\bigr\}\footnote{The invariant $\mathrm{Cliff}_n(C)$ is denoted in the paper \cite{LN1} by $\gamma_n'(C)$.  Since the appearance of \cite{LN1}, it has become abundantly clear that $\mathrm{Cliff}_n(C)$, defined as above, is the most relevant Clifford type invariant for rank $n$ vector bundles on $C$. Accordingly, the  notation $\mathrm{Cliff}_n(C)$ seems appropriate.}.$$
Note that $\mathrm{Cliff}_1(C)=\mathrm{Cliff}(C)$ is the classical Clifford index of $C$. By specializing to sums of line bundles, it is easy to check that $\mathrm{Cliff}_n(C)\leq \mathrm{Cliff}(C)$ for all $n\geq 1$. Mercat \cite{Me} proposed the following interesting conjecture, which we state in the form of \cite{LN1} Conjecture 9.3, linking the newly-defined invariants $\mathrm{Cliff}_n(C)$ to the classical geometry of $C$:
$$
(M_n):  \  \   \mathrm{Cliff}_n(C)=\mathrm{Cliff}(C).
$$
Mercat's conjecture $(M_2)$ holds for various classes of curves, in particular general $k$-gonal curves of genus $g>4k-4$, or arbitrary smooth plane curves, see \cite{LN1}. In \cite{FO} Theorem 1.7, we have verified $(M_2)$ for a general curve $[C]\in \cM_g$ with $g\leq 16$. More generally, the statement
$(M_2)$ is a consequence of the \emph{Maximal Rank Conjecture} (see \cite{FO} Conjecture 2.2), therefore it is expected to be true for a general curve $[C]\in \cM_g$.  However, for every genus $g\geq 11$ there exist curves $[C]\in \cM_g$ with maximal Clifford index $\mathrm{Cliff}(C)=[\frac{g-1}{2}]$ carrying stable rank $2$ vector bundles $E$ with $h^0(C, E)=4$ and $\gamma(E)<\mathrm{Cliff}(C)$, see \cite{FO} Theorems 3.6 and 3.7 and \cite{LN2} Theorem 1.1 for an improvement. For these curves, the inequality $\mathrm{Cliff}_2(C)<\mathrm{Cliff}(C)$ holds.
\vskip 3pt
Obvious questions emerging from this discussion are whether such results are specific to (i) rank $2$ bundles with $4$ sections,  or to (ii) curves with maximal Clifford index $[\frac{g-1}{2}]$. First we prove that under general circumstances, curves on $K3$ surfaces carry rank $2$ vector bundles $E$ with a prescribed (and exceptionally high) number of sections invalidating Mercat's inequality $\gamma(E)\geq \mathrm{Cliff}(C)$:
\begin{theorem}\label{existence1}
We fix integers $p\geq 1$ and $a\geq 2p+3$. There exists a smooth curve $C$ of genus $2a+1$ and Clifford index $\mathrm{Cliff}(C)=a$, lying on a $K3$ surface $C\subset S\subset \PP^{2p+2}$ with $\mathrm{Pic}(S)=\mathbb Z\cdot C\oplus \mathbb Z\cdot H$, where $H^2=4p+2,\  H\cdot C=\mathrm{deg}(C)=2a+2p+1$, as well as a stable rank $2$ vector bundle $E\in \mathcal{SU}_C\bigl(2, \OO_C(H)\bigr)$, such that $h^0(C, E)=p+3$.
In particular $\gamma(E)=a-\frac{1}{2}<\mathrm{Cliff}(C)$ and Mercat's conjecture $(M_2)$ fails for $C$.
\end{theorem}
It is well-known cf. \cite{M2}, \cite{V1}, that a curve $[C]\in \cM_{2a+1}$ lying on a $K3$ surface $S$ possesses a rank $2$ vector bundle
$F\in \mathcal{SU}_C(2, K_C)$ with $h^0(C, F)=a+2$. In particular, $\gamma(F)=a\geq \mathrm{Cliff}(C)$ (with equality if $\mathrm{Pic}(S)=\mathbb Z\cdot C$), hence such bundles satisfy condition $(M_2)$. Let us consider  the $K3$ locus in the moduli space of curves
$$\K_g:=\{[C]\in \cM_g: C \ \mbox{lies on a } K3 \mbox{ surface}\}.$$
When $g=11$ or $g\geq 13$, the variety $\K_g$ is irreducible and $\mbox{dim}(\K_g)=19+g$, see \cite{CLM} Theorem 5.
For integers $r, d\geq 1$ such that $d^2>4(r-1)g$ and $2r-2\nmid d$, we define the \emph{Noether-Lefschetz} divisor inside the locus of sections of $K3$ surfaces
$$
\mathfrak{NL}_{g, d}^r:=\left\{[C]\in \K_g \left|
\begin{array}{l}
C\mbox{ lies on a } K3 \mbox{ surface } S,  \mbox{  }\
\mathrm{Pic}(S)\supset \mathbb Z\cdot C\oplus \mathbb Z\cdot H,\\
H\in \mathrm{Pic}(S) \mbox{ is nef},
 H^2=2r-2, \; \ C\cdot H=d, \;  \ C^2=2g-2 \end{array}
\right. \right\}.
$$
A consequence of Theorem \ref{existence1} can be formulated as follows:
\begin{corollary}
We fix integers $p\geq 1$ and $a\geq 2p+3$ and set $g:=2a+1$. Then Mercat's conjecture $(M_2)$ fails generically along the Noether-Lefschetz locus
$\mathfrak{NL}_{g, 2a+2p+1}^{2p+2}$ inside $\K_g$, that is,
$\mathrm{Cliff}_2(C)<\mathrm{Cliff}(C)$ for a general point $[C]\in \mathfrak{NL}_{g, 2a+2p+1}^{2p+2}$.
\end{corollary}

It is natural to wonder whether it is necessary to pass to a Noether-Lefschetz divisor in $\K_g$, or perhaps,  all curves $[C]\in \K_g$ give counterexamples to conjecture $(M_2)$. To see that this is not always the case and all conditions in Theorem \ref{existence1} are necessary, we study in detail the case $g=11$. Mukai  \cite{M3} proved that a general curve $[C]\in \cM_{11}$ lies on a unique $K3$ surface $S$ with $\mathrm{Pic}(S)=\mathbb Z\cdot C$, thus, $\cM_{11}=\K_{11}$.

\begin{theorem}\label{gen11}
For a general curve $[C]\in \cM_{11}$ one has the equality $\mathrm{Cliff}_2(C)=\mathrm{Cliff}(C)$, that is, Mercat's conjecture holds generically on $\cM_{11}$. Furthermore, the locus
$$\{[C]\in \cM_{11}: \mathrm{Cliff}_2(C)<\mathrm{Cliff}(C)\}$$ can be identified with the Noether-Lefschetz divisor
$\mathfrak{NL}_{11, 13}^4$ on $\cM_{11}$.
\end{theorem}

In Section 5, we describe in detail the divisor $\mathfrak{NL}_{11, 13}^4$ and discuss, in connection with Mercat's conjecture, the action of the Fourier-Mukai involution $FM:\F_{11}\rightarrow \F_{11}$ on the moduli space of polarized $K3$ surfaces of genus $11$. The automorphism $FM$ acts on the set of Noether-Lefschetz divisors and in particular it (i) fixes the $6$-gonal locus $\cM_{11, 6}^1$ and it maps the divisor $\mathfrak{NL}_{11, 13}^4$ which corresponds to certain elliptic $K3$ surfaces, to the Noether-Lefschetz divisor corresponding to $K3$ surfaces carrying a rational curve of degree $3$.


\vskip 4pt
Next we turn our attention to the conjecture $(M_n)$ for $n\geq 3$. It was observed in \cite{LMN} that Mukai's description \cite{M4} of a general curve of genus $9$ in terms of linear sections of a certain rational homogeneous variety, and especially the connection to rank $3$ Brill-Noether theory, can be used to construct,  on a general curve $[C]\in \cM_9$, a stable vector bundle
$E\in \mathcal{SU}_C(3, K_C)$ such that $h^0(C, E)=6$. In particular $\gamma(E)=\frac{10}{3}<\mathrm{Cliff}(C)$, that is, Mercat's conjecture $(M_3)$ fails for a general curve $[C]\in \cM_9$. A similar construction is provided in \cite{LMN} for a general curve of genus $11$. In what follows we outline a construction illustrating  that the results from \cite{LMN} are part of a larger picture and curves on  $K3$ surfaces carry vector bundles $E$ of rank at least $3$ with $\gamma(E)<\mathrm{Cliff}(C)$.
\vskip 4pt

 Let $S$ be a $K3$ surface and $C\subset S$ a smooth curve of genus $g$. We choose a linear series $A\in W^r_d(C)$ of minimal degree such that the Brill-Noether number $\rho(g, r, d)$ is non-negative, that is, $d:=r+[\frac{r(g+1)}{r+1}]$. The \emph{Lazarsfeld bundle} $M_A$ on $C$ is defined as the kernel of the evaluation map, that is,
$$ 0\longrightarrow M_A\longrightarrow H^0(C, A)\otimes \OO_C\stackrel{\mathrm{ev}_C}\longrightarrow A\longrightarrow
0.$$ As usual, we set $Q_A:=M_A^{\vee}$, hence $\mbox{rank}(Q_A)=r$ and $\mbox{det}(Q_A)=A$.
Following a procedure that already appeared  in  \cite{L}, \cite{M2}, \cite{V1}, we note that $C$ carries a vector bundle of rank $r+1$ with canonical
determinant and unexpectedly many global sections:

\begin{theorem}\label{K3bundles}
For a curve $C\subset S$ and $A\in W^r_d(C)$ as above there exists a globally generated vector bundle
$E$ on $C$ with $\mathrm{rank}(E)=r+1$ and $\mathrm{det}(E)=K_C$, expressible as an extension
$$0\longrightarrow Q_A\longrightarrow E\longrightarrow K_C\otimes A^{\vee}\longrightarrow 0,$$ satisfying the condition \ $h^0(C, E)=h^0(C, A)+h^0(C, K_C\otimes A^{\vee})=g-d+2r+1.$ If moreover $r\leq 2$ and $\mathrm{Pic}(S)=\mathbb Z\cdot C$, then the above extension is non-trivial.

\end{theorem}
When $r=1$ the rank $2$ bundle $E$ constructed in Theorem \ref{K3bundles} is well-known
and plays an essential role in \cite{V1}. In this case $\gamma(E)\geq [\frac{g-1}{2}]$. For $r=2$ and $g=9$ (in which case $A\in W^2_8(C)$), or for $g=11$ (and then $A\in W^2_{10}(C)$), Theorem \ref{K3bundles} specializes to the construction in \cite{LMN}.  When $\mathrm{rank}(E)=3$,  we observe by direct calculation that $\gamma(E)<[\frac{g-1}{2}]$. In view of providing counterexamples to Mercat's conjecture $(M_3)$, it is thus important to determine whether $E$ is stable.

\begin{theorem}\label{stability}
Fix $C\subset S$ as above with $g=7, 9$ or $g\geq 11$ such that $\mathrm{Pic}(S)=\mathbb Z\cdot C$, as well as $A\in W^2_d(C)$, where $d:=[\frac{2g+8}{3}]$. Then any globally generated rank $3$ vector bundle $E$ on $C$ lying non-trivially in the extension
$$0\longrightarrow Q_A\longrightarrow E\longrightarrow K_C\otimes A^{\vee}\longrightarrow 0,$$
and with $h^0(C, E)=h^0(C, A)+h^0(C, K_C\otimes A^{\vee})=g-d+5$, is stable.
\end{theorem}
As a corollary, we note that for sufficiently high genus Mercat's statement $(M_3)$ fails to hold for \emph{any} smooth curve of maximal Clifford index lying on a $K3$ surface.
\begin{corollary}
We fix an integer $g=9$ or $g\geq 11$ and a curve $[C]\in \K_g$. Then the inequality $\mathrm{Cliff}_3(C)<[\frac{g-1}{2}]$ holds. In particular, Mercat's conjecture $(M_3)$ fails generically along $\K_g$.
\end{corollary}
\vskip 3pt
We close the Introduction by thanking Herbert Lange and Peter Newstead for making a number of very pertinent comments on the first version of this paper.

\section{Higher rank vector bundles with canonical determinant}

 In this section we treat Mercat's conjecture $(M_3)$ and prove Theorems \ref{K3bundles} and \ref{stability}. We begin with a curve $C$ of genus $g$ lying on a smooth $K3$ surface $S$ such that $\mathrm{Pic}(S)=\mathbb Z\cdot C$, and fix a linear series $A\in W^2_d(C)$
 of minimal degree $d:=[\frac{2g+8}{3}]$. Under such assumptions both $A$ and $K_C\otimes A^{\vee}$ are base point free.  From the onset, we  point out that the existence of vector bundles of higher rank on $C$ having exceptional Brill-Noether behaviour has been repeatedly used in
\cite{L}, \cite{M2} and \cite{V1}. Our aim is to study these bundles from the point of view of Mercat's conjecture and discuss their stability.

We define the \emph{Lazarsfeld-Mukai} sheaf $\F_A$ via
the following exact sequence on $S$:
$$0\longrightarrow \F_A\longrightarrow H^0(C, A)\otimes \OO_S \stackrel{\mathrm{ev}_S}\longrightarrow A\longrightarrow 0.$$
Since $A$ is base point free, $\F_A$ is locally free. We consider the vector bundle $\E_A:=\F_A^{\vee}$ on $S$,
which by dualizing, sits in an exact sequence
\begin{equation}\label{e1}
0\longrightarrow H^0(C, A)^{\vee}\otimes \OO_S\longrightarrow \E_A\longrightarrow K_C\otimes A^{\vee}\longrightarrow 0.
\end{equation}
Since $K_C\otimes A^{\vee}$ is assumed to be base point free, the bundle $\E_A$ is globally generated. It is well-known (and follows from the sequence (\ref{e1}), that $c_1(\E_A)=\OO_S(C)$ and $c_2(\E_A)=d$.

\vskip 3pt
\noindent \emph{Proof of Theorem \ref{K3bundles}.} We write down the following commutative diagram
$$\begin{array}{ccccccccc}
& \; &\;  0 & \; & 0 & \;&   \\
 & \; & \rmapdown{} &\; & \rmapdown{}\\
& \;  & H^0(C, A)\otimes \OO_S(-C) & \stackrel{=}\longrightarrow & H^0(C, A)\otimes \OO_S(-C) & \; &\;   \\
 & \; & \rmapdown{} & \; & \rmapdown{} & \; & \; & \\
   0 & \longrightarrow & \F_A & \longrightarrow  & H^0(C, A)\otimes \OO_S & \longrightarrow & A & \longrightarrow &
   0 \\
  & \; &\rmapdown{} & \; &\rmapdown{} & \; &\rmapdown{=} &  \; & \\
 0 &\longrightarrow & M_A &  \longrightarrow & H^0(C, A)\otimes \OO_C & \longrightarrow &
A &\longrightarrow & 0 \\
 &\; & \rmapdown{} & \; & \rmapdown{}\\
 & \;&  0 & \;&  0\\
\end{array}$$
from which, if we set $F_A:=\F_A\otimes \OO_C$ and $E_A:=\E_A\otimes \OO_C$, we obtain the exact sequence
$$0\longrightarrow M_A\otimes K_C^{\vee}\longrightarrow H^0(C, A)\otimes K_C^{\vee}\longrightarrow F_A\longrightarrow M_A\longrightarrow 0$$
(use that $\mbox{Tor}_{\OO_S}^1(M_A, \OO_C)=M_A\otimes K_C^{\vee}$). Taking duals, we find the exact sequence
\begin{equation}\label{e2}
0\longrightarrow Q_A\longrightarrow E_A\longrightarrow K_C\otimes A^{\vee}\longrightarrow 0.
\end{equation}
Since $S$ is regular, from (\ref{e1}) we obtain that $h^0(S, \E_A)=h^0(C, A)+h^0(C, K_C\otimes A^{\vee})$
while $H^0(S, \E_A\otimes \OO_S(-C))=0$, that is, $$h^0(S, \E_A)\leq h^0(C, E_A)\leq h^0(C, A)+h^0(C, K_C\otimes A^{\vee}).$$ Thus the sequence (\ref{e2}) is exact on global sections.
\vskip 4pt
We are left with proving that the extension (\ref{e2}) is non-trivial. We set $r=2$ and then $\mbox{rank}(\E_A)=3$ and place ourselves in the situation when $\mathrm{Pic}(S)=\mathbb Z\cdot C$ (the case $r=1$ works similarly). By contradiction we assume that $E_A=Q_A\oplus (K_C\otimes A^{\vee})$ and denote by $s:E_A\rightarrow Q_A$ a retract and by $\tilde{s}:\E_A\rightarrow Q_A$ the induced map.
We set $\cM:=\mbox{Ker}\{\E_A\stackrel{\tilde{s}}\longrightarrow Q_A\}$, hence $\cM$ can be regarded as an elementary transformation of the Lazarsfeld-Mukai bundle $\E_A$ along $C$. By direct calculation we find that
$$c_1(\cM)=\OO_S(-C) \mbox{ and } \ c_2(\cM)=2d-2g+2,$$ hence the discriminant of $\cM$ equals
$\Delta(\cM):=6c_2(\cM)-2c_1^2(\cM)=4(3d-4g+4)<0.$ Thus the sheaf $\cM$ is $\OO_S(C)$-unstable. Applying \cite{HL} Theorems 7.3.3 and 7.3.4, there exists a subsheaf $\cM'\subset \cM$ such that if $\xi_{\cM, \cM'}:=\frac{c_1(\cM')}{\mathrm{rank}(\cM')}-\frac{c_1(\cM)}{\mathrm{rank}(\cM)} \in \mathrm{Pic}(S)_{\mathbb R}$, then
$$(i) \ \ \xi_{\cM, \cM'}\cdot C>0 \ \ \mbox{ and } \ (ii)\ \mbox{ } \xi^2_{\cM, \cM'}\geq -\frac{\Delta(\cM)}{18}.$$
Since $\mathrm{Pic}(S)=\mathbb Z\cdot C$, we may write $c_1(\cM')=\OO_S(aC)$ and also set $r':=\mathrm{rank}(\cM')$. The Lazarsfeld-Mukai bundle $\E_A$ is
$\OO_S(C)$-stable, in particular $\mu_C(\cM')\leq \mu_C(\E_A)$, which yields $a\leq 0$. Then from $(i)$ we write that
$0\leq \frac{a}{r'}+\frac{1}{3}\leq \frac{1}{3}$, whereas from $(ii)$ one finds
$$\frac{1}{9}\geq \frac{4(g-1)-3d}{9(g-1)}\Leftrightarrow d\geq g-1,$$
which is a contradiction. It follows that the extension (\ref{e2}) is non-trivial.
\hfill $\Box$

\vskip 4pt
It is  natural to ask when is the above constructed bundle $E_A$ stable. We give an affirmative answer under certain generality assumptions, when $r<3$.

We fix a $K3$ surface $S$ such that $\mbox{Pic}(S)=\mathbb Z\cdot C$ and as before, set $d:=[\frac{2g+8}{3}]$. Under these assumptions, it follows from \cite{L} that  $C$ satisfies the Brill-Noether theorem. We prove the stability of  every globally generated non-split bundle $E$ sitting in an extension of the form (\ref{e2}) and having a maximal number of sections.
\vskip 4pt
\noindent
\emph{Proof of Theorem \ref{stability}.}
We first discuss the possibility of a destabilizing sequence $$0\longrightarrow F\longrightarrow E\longrightarrow B\longrightarrow 0,$$ where
$F$ is a vector bundle of rank $2$ and $\mbox{deg}(F)\geq\frac{4}{3}(g-1)$. Since $E$ is globally generated, it follows that $B$ is globally generated as well, hence $h^0(C, B)\geq 2$, in
particular $\mbox{deg}(B)\geq (g+2)/2$ and hence $\mbox{deg}(F)\leq \frac{3}{2}g-3$. Since $\mbox{deg}(B)\leq \frac{2}{3}(g-1)$ and $C$ is Brill-Noether general, it follows that $h^0(C, B)=2$, therefore $h^0(C, F)\geq g-d+3$. There are two cases to distinguish, depending on whether $F$ possesses a subpencil or not.

Assume first that $F$ has no subpencils. We apply \cite{PR} Lemma 3.9 to find that $h^0(C, \mbox{det}(F))\geq 2h^0(C, F)-3\geq 2g-2d+3$.
Writing down the inequality $$\rho\bigl(g, 2g-2d+2, \mbox{deg}(F)\bigr)\geq 0$$ and using that $\mbox{deg}(F)<\frac{3}{2}g-3$, we obtain a contradiction. If on the other hand, $F$ has a subpencil, then as pointed out in \cite{FO} Lemma 3.2, $\gamma(F)\geq \mbox{Cliff}(C)$, but again this is a contradiction. This shows that $E$ cannot have a rank $2$ destabilizing subsheaf.

We are left with the possibility of a destabilizing short exact sequence
$$0\longrightarrow B\longrightarrow E \longrightarrow F\longrightarrow 0,$$
where $B$ is a line bundle with $\mbox{deg}(B)\geq \frac{2}{3}(g-1)$ and $F$ is a rank $2$ bundle.  The bundle $Q_A$ is well-known to be stable and based on slope considerations,  $B$ cannot be a subbundle of $Q_A$, that is, necessarily $H^0(C, K_C\otimes A^{\vee}\otimes B^{\vee})\neq 0$. Since the bundle $E$ is not decomposable, it follows that $\mbox{deg}(B)\leq \mbox{deg}(K_C\otimes A^{\vee})-1=2g-3-d$. Furthermore $h^1(C, B)\geq 3$.

If $F$ is not stable, we reason along the lines of \cite{LMN} Proposition 3.5
and pull-back a destabilizing line subbundle of $F$ to obtain a rank $2$ subbundle $F'\subset E$ such that $$\mbox{deg}(F')\geq \mbox{deg}(B)+\frac{1}{2}\Bigl(\deg(E)-\mbox{deg}(B)\Bigr)\geq \frac{4}{3}(g-1),$$
which is the case which we have already ruled out. So we may assume
that $F$ is stable.
We write $h^0(C, B)=a+1$, hence $h^0(C, F)\geq g-d-a+4$. Assume first that $F$ admits no subpencils. Then from \cite{PR} Lemma 3.9 we find the following estimate for the number of sections of the line bundle \ $\mbox{det}(F)=K_C\otimes B^{\vee}$,
$$h^0(C, K_C\otimes B^{\vee})\geq 2h^0(C, F)-3\geq 2g-2d-2a+5,$$
which, after applying Riemann-Roch to $B$, leads to  the inequality
$$3a\geq g-2d+5+\mbox{deg}(B).$$
Combining this estimate with the Brill-Noether inequality $\rho\bigl(g, a, \mbox{deg}(B)\bigr)\geq 0$ and substituting the actual value of $d$, we find that $3a+3\geq g$. On the other hand $a\leq h^0(C, K_C\otimes A^{\vee})-2=g-d<\frac{g-3}{3}$, and this is a contradiction.

Finally, if $F$ admits a subpencil, then $\gamma(F)\geq \mathrm{Cliff}(C)$. Combining this with the classical Clifford inequality for $B$, we find that $\gamma(E)\geq \mathrm{Cliff}(C)$, which again is a contradiction. We conclude that the rank $3$ bundle $E$ must be stable.
\hfill  $\Box$

\section{Rank $2$ bundles and Koszul classes}

The aim of this section is to prove Theorem \ref{existence1}. We shall construct rank $2$ vector bundles on curves using a connection between vector bundles on curves and Koszul cohomology of line bundles, cf.  \cite{AN} and \cite{V2}. Let us  recall that for a smooth projective variety $X$, a sheaf $\F$  and a globally generated line bundle $L$ on $X$, the Koszul cohomology group $K_{p, q}(X; \F, L)$ is defined as the cohomology of the complex:
$$\bigwedge^{p+1} H^0(L)\otimes H^0(\F\otimes L^{q-1})\stackrel{d_{p+1, q-1}}\longrightarrow \bigwedge^p H^0(L)\otimes H^0(\F\otimes L^q) \stackrel{d_{p, q}}\longrightarrow \bigwedge^{p-1}H^0(L)\otimes H^0(\F\otimes  L^{q+1}).$$
Most of the time $\F=\OO_X$, and then one writes $K_{p, q}(X; \OO_X, L):=K_{p, q}(X, L)$.

A Koszul class $[\zeta]\in K_{p, 1}(X, L)$ is said to have rank $\leq n$, if there exists a subspace $W\subset H^0(X, L)$ with $\mbox{dim}(W)=n$ and a representative $\zeta \in \wedge^p W\otimes H^0(X, L)$. The smallest number $n$ with this property is the rank of the syzygy $[\zeta]$.
\vskip 3pt

Next we discuss a connection  due to Voisin \cite{V2} and expanded in \cite{AN},  between rank $2$ vector bundles on curves and syzygies. Let $E$ be a rank $2$ bundle on a smooth curve $C$ with $h^0(C, E)\geq p+3\geq 4$ and set $L:=\mbox{det}(E)$. Let
$$\lambda:\wedge^2 H^0(C, E)\rightarrow H^0(C, L)$$ be the determinant map, and we assume that there exists linearly independent sections $e_1\in H^0(C, E)$ and
$e_2, \ldots, e_{p+3}\in H^0(C, E)$, such that the map
$$\lambda\bigl(e_1\wedge \ - \ \bigr):\langle e_{2}, \ldots, e_{p+3}\rangle\rightarrow H^0(C, L)$$
in injective onto its image. Such an assumption is automatically satisfied for instance if $E$ admits no subpencils.
We introduce the subspace
$$W:=\bigl\langle s_2:=\lambda(e_1\wedge e_2), \ldots, s_{p+3}:=\lambda(e_1\wedge e_{p+3})\bigr\rangle
\subset H^0(C, L).$$ By assumption, $\mbox{dim}(W)=p+2$. Following \cite{AN} and \cite{V2}, we define the tensor
$$
\zeta(E):=\sum_{i<j} (-1)^{i+j}\  s_2\wedge \ldots \wedge \hat{s_i}\wedge \ldots \wedge \hat{s_j} \wedge \ldots \wedge s_{p+3}\otimes \lambda(e_i\wedge e_j) \in \wedge^p W \otimes H^0(C, L).
$$
One checks that $d_{p, 1}(\zeta(E))=0$, hence  $[\zeta(E)]\in K_{p, 1}(C, L)$ is a non-trivial Koszul class of rank at most $p+2$.
Conversely, starting with a non-trivial class $[\zeta]\in K_{p, 1}(C, L)$ represented by an element $\zeta$ of $\wedge^p W\otimes H^0(C, L)$ where $\mbox{dim}(W)=p+2$, Aprodu and Nagel \cite{AN} Theorem 3.4 \ constructed a rank $2$ vector bundle $E$ on $C$ with $\mbox{det}(E)=L$, \  $h^0(C, E)\geq p+3$ and such that $[\zeta(E)]=[\zeta]$. This correspondence sets up a dictionary between the Brill-Noether loci in $\{E\in \mathcal{SU}_C(2, L): h^0(C, E)\geq p+3\}$  and Koszul classes of rank at most $p+2$ in $K_{p, 1}(C, L)$.
\vskip 4pt

Let us now fix integers $p\geq 1$ and $a \geq 2p+3$. Using the surjectivity of the period mapping, see e.g.  \cite{K} Theorem 1.1, one can construct a smooth $K3$ surface $S\subset \PP^{2p+2}$  of
degree $4p+2$  containing a smooth curve $C\subset S$ of degree $d:=2a+2p+1$ and genus $g:=2a+1$. The
surface $S$ can be chosen with $\mbox{Pic} (S) = \ZZ\cdot H \oplus \ZZ\cdot C $, where  $H^2= 4p+2$, $H\cdot C= d$ and $C^2= 4a$. The smooth curve $H\subset C$ is the hyperplane section of $S$ and has genus $g(H)=2p+2$. The following observation is trivial:

\begin{lemma}\label{vanishing}
Keeping the notation above, we have that $H^0(S, \OO_S(H-C))=0$.
\end{lemma}
\begin{proof} It is enough to notice that $H$ is nef and $(H-C)\cdot H=2p-2a+1<0$.
\end{proof}

We consider the decomposable rank $2$ bundle $K_H=A\oplus (K_H\otimes A^{\vee})$ on $H$, where $A\in W^1_{p+2}(H)$. Via the Green-Lazarsfeld non-vanishing theorem \cite{GL} (or equivalently, applying \cite{AN}), one obtains a non-zero Koszul class of rank $p+1$
$$\beta:=\Bigl[\zeta\bigl(A\oplus (K_H\otimes A^{\vee})\bigr)\Bigr]\in K_{p, 1}(H, K_H).$$
Since $S$ is a regular surface, there exist an exact sequence
$$0\longrightarrow H^0(S, \OO_S)\longrightarrow H^0(S, \OO_S(H))\longrightarrow H^0(H, K_H)\longrightarrow 0,$$
which induces an isomorphism \cite{G} Theorem (3.b.7)
$$\mathrm{res}_{H}:K_{p, 1}(S, \OO_S(H))\cong K_{p, 1}(H, K_H).$$ By construction,  the non-trivial class $\alpha:=\mathrm{res}_{H}^{-1}(\beta)\in K_{p, 1}(S, \OO_S(H))$ has rank  at most $\mbox{rank}(\beta)+1 =p+2$. Using \cite{G} Theorem (3.b.1), we write the following exact sequence in Koszul cohomology:
$$\cdots \rightarrow K_{p, 1}(S; -C, H)\rightarrow K_{p, 1}(S, H)\rightarrow K_{p, 1}(C, H\otimes \OO_C)\rightarrow K_{p-1, 2}(S; -C, H)\rightarrow\cdots. $$
Since $H^0(S, \OO_S(H-C))=0$, it follows that $K_{p, 1}(S;-C, H)=0$, in particular the non-zero class $\alpha\in K_{p, 1}(S, H)$ can be viewed as a Koszul class of rank at most $p+2$ inside the group $K_{p, 1}(C, \OO_C(H))$. This class corresponds to a \emph{stable} rank $2$ bundle on $C$:

\begin{proposition}
Let $C\subset S\subset \PP^{2p+2}$ as above and $L:=\OO_C(1)\in \mathrm{Pic}^{2a+2p+1}(C)$. Then there exists a stable vector bundle $E\in \mathcal{SU}_C(2, L)$ with $h^0(C, E)=p+3$.
\end{proposition}
\begin{proof}
From \cite{AN} we know that there exists a rank $2$ vector bundle $E$ on $C$ with $\mbox{det}(E)=L$ such that $[\zeta(E)]=\alpha\in K_{p, 1}(C, L)$, in particular $h^0(C, E)\geq p+3$. Geometrically, $E$ is the restriction to $C$ of the Lazarsfeld-Mukai bundle $\E_A$ on $S$ corresponding to a pencil $A\in W^1_{p+2}(H)$. In particular, $E$ is globally generated, being the restriction of a globally generated bundle on $S$. We also know that $\mbox{Cliff}(C)=a$ (to be proved in Proposition \ref{chapo1}). Since $\gamma(E)\leq a-\frac{1}{2}<\mathrm{Cliff}(C)$, it follows that $E$ admits no subpencils (If $B\subset E$ is a subpencil, then $h^0(C, L\otimes B^{\vee})\geq 2$ because $E$ is globally generated. It is easily verified that both $B$ and $L\otimes B^{\vee}$ contribute to $\mathrm{Cliff}(C)$, which brings about a contradiction). Assume now that
$$0\longrightarrow B\longrightarrow E\longrightarrow L\otimes B^{\vee}\longrightarrow 0$$
is a destabilizing sequence, where $B\in \mathrm{Pic}(C)$ has degree at least $a+p+1$. As already pointed out, $h^0(C, B)\leq 1$, hence
$h^0(C, L\otimes B^{\vee})\geq p+2$. If $h^1(C, L\otimes B^{\vee})\leq 1$, then $p+2\leq h^0(C, L\otimes B^{\vee})\leq 1+\mbox{deg}(L\otimes B^{\vee})-2a$, which leads to a contradiction. If on the other hand $h^1(C, L\otimes B^{\vee})\geq 2$, then $\mathrm{Cliff}(L\otimes B^{\vee})\leq a-p-2<a$, which is impossible. Thus $E$ is a stable vector bundle.
\end{proof}

\vskip 3pt

We are left with showing that the curve $C\subset S$ constructed above has maximal Clifford index $a$. Note that the corresponding statement when $p=1$ has been proved in \cite{FO} Theorem 3.6.

\begin{proposition}\label{chapo1}
We fix integers $p\geq 1$, $a\geq 2p+3$ and a $K3$ surface $S$ with Picard lattice $\mathrm{Pic}(S)=\mathbb Z\cdot H\oplus \mathbb Z\cdot C$ where $C^2=4a$, \ $H^2=4p+2$ and $C\cdot H=2a+2p+1$. Then  $\mathrm{Cliff}(C) =a$.
\end{proposition}
\begin{proof}
 First note that $C$ has Clifford dimension $1$, for curves $C\subset S$ of higher Clifford dimension have even genus. Observe also that $h^0(C, \OO_{C}(1)) = 2p+3$ and $h^1(C, \OO_C (1))= 2$, hence  $\OO_C(1)$ contributes to the Clifford index of $C$ and
$$
\Cliff (C) \leq \Cliff(C, \OO(1) ) = C\cdot H - 2(2p+2) = 2a -2p-3\  \bigl(\geq a\bigr).
$$
Assume by contradiction that $\mathrm{Cliff}(C) < a $. According to \cite{GL2},  there exists an effective divisor $D\equiv mH+nC$ on $S$  satisfying the conditions
\begin{equation}\label{cond2}
h^0(S,\OO_S(D)) \geq 2, \ \  h^0(S,\OO_S(C-D)) \geq 2, \\ \  C\cdot D \leq g-1,
\end{equation} and with $\mathrm{Cliff}(\OO_C(D))=\mathrm{Cliff}(C)$.  By \cite{M} Lemma 2.2, the dimension $h^0(C', \OO_{C'}(D))$ stays constant for all smooth curves $C'\in |C|$ and its value equals $h^0(S,D)$. We conclude that $\Cliff (C) = \Cliff (\OO_C(D)) = D\cdot C - 2 \dim |D|$.
We summarize the  numerical consequences of the inequalities (\ref{cond2}):
$$
\begin{array}{rl}
(i) & md + 2n (g-1) \leq g-1 \\
(ii) & (2p+1)m^2 +mnd + n^2(g-1) \geq 0\\
(iii) & (4p+2)m + dn >2,
\end{array}
$$
We claim that for any divisor $D\subset S$ verifying $(i)$-$(iii)$, the following inequality holds:
$$
\Cliff(\OO_C(D)) = D\cdot C - D^2 -2 \geq H\cdot C -H^2 - 2 = 2a-2p - 3 \geq a.
$$
This will contradict the assumption $\Cliff(C) < a$. The proof proceeds along the lines of Theorem 3 in \cite{F1}, with
the difference that we must also consider curves with $D^2=0$, that is, elliptic pencils which we now characterize. By direct calculation, we note  that there are no $(-2)$-curves in $S$.
Equality holds in $(ii)$ when $m=-n$ or $m=-un$ with $u:= 2a/(2p+1)$. \\
First, we describe the effective divisors $D\subset S$ with self-intersection $D^2 =0$. Consider the case $m=-un$. If $2p +1$ does not divide $a$,
then $D\equiv 2aH - (2p+1)C$ and $D\cdot C = 2a(2a -2p-1) > g-1$, that is,  $D$ does not verify condition $(i)$. If $ a= k(2p + 1)$, for $k\geq 2$,  then $D\equiv 2kH-C$.  Notice that $D\cdot C = a(4k-4) + 2k(2p+1) > 2a $ for $k \geq 2$, that is,
$D$  does not satisfies $(i)$.\\
In the the case $m= -n$, the effective divisor $D\equiv C-H$, satisfies $(i)$-$(iii)$ and
$$
\Cliff(\OO_C(C-H))  = 2a-2p-3 \geq a.
$$

{\it Case $n<0$}. From $(ii)$ we have either $m < -n$ or $m >-un$. In the first case, by using inequality $(iii)$, we obtain
$
2< -(4p+2)n +d n = n (2a-2p-1)$,
which is a contradiction since $n<0$ and $2a> 2p +1$. Suppose $m > -un>0$. Inequality $(i)$ implies that
$$
(-n)\frac{2ad}{2p+1} < -(g-1) (2n-1) = -2a(2n-1),
$$
then $ (-n) (d- (4p+2)) < 2p+1 $ and since $d> 4p+2$, this yields $2a+2p+1=d < 6p+3$ which contradicts the hypothesis $a \geq 2p+3$.\\

{\it Case $n>0$ }.  Again, by condition $(ii)$, we have either that $m<-un$ or $m> -n$. In the first case, using $(iii)$ we write that
$$
0< (4p+2)m +dn < n \left( d - (4p+2)\frac{2a}{2p+1} \right),
$$
but one can get easily check that $d (2p+1)< 2a(4p+2) $, which yields a contradiction. Suppose now $-n<m<0$.
By $(i)$ we have $2a(2n-1) \leq  -md <  nd$, so
$
n < \frac{2a}{ 4a-d} =  \frac{2a}{2a-2p-1}< 2,
$
since $a \geq 2p+1$. This implies $n=1$, therefore for $n>0$ there are no divisors $D\subset S$ with $D^2>0$ satisfying the inequalities $(i)$-$(iii)$.
\\

{\it Case $n=0$}. From $(i)$, one writes $m \leq \frac{g-1}{d} = \frac{2a}{2a+2p+1} <1$, but this yields to a contradiction since
by $(iii)$ it follows that  $m>0$. The proof is thus finished.

\end{proof}

\section{Curves with prescribed gonality and small rank $2$ Clifford index}

The equality $\mathrm{Cliff}_2(C)= \mathrm{Cliff}(C)$ is known to be valid for \emph{arbitrary} $k$-gonal curves $[C]\in \cM_{g, k}^1$ of genus $g>(k-1)(2k-4)$. It is thus of some interest to study Mercat's question for arbitrary curves in a given gonality stratum in $\cM_g$
and decide how sharp is this quadratic bound. We shall construct curves $C$ of unbounded genus and relatively small gonality, carrying a stable rank $2$ vector bundle
$E$ with $h^0(C, E)=4$ such that $\gamma(E)<\mathrm{Cliff}(C)$. In order to be able to determine the gonality of $C$, we realize it as a section of a $K3$ surface $S$ in $\PP^4$ which is special in the sense of Noether-Lefschetz theory. The pencil computing the gonality is the restriction of an elliptic pencil on the surface. The constraint of having a Picard lattice of rank $2$ containing, apart from the hyperplane class, both an elliptic pencil and a curve $C$ of prescribed genus, implies that the discriminant of $\mathrm{Pic}(S)$ must be a perfect square. This imposes severe restrictions on the genera for which such a construction could work.



\begin{theorem} \label{K3surf}
We fix integers  $a\geq 3$ and $b=4,5,6$. There exists a smooth curve $C\subset \PP^4$ with
$$\mathrm{deg}(C)=
6a+b, \ \ g(C)=3a^2+ ab +1 \ \mbox{ and gonality }\ \mathrm{gon}(C)=ab,$$ such that $C$ lies on a $(2, 3)$ complete intersection $K3$ surface.
In particular $K_{1, 1}(C, \OO_C(1))\neq 0$ and conjecture $(M_2)$ fails for $C$.
\end{theorem}

Before presenting the proof, we discuss the connection between  Theorem \ref{K3surf} and conjecture $(M_2)$. For $C\subset S\subset \PP^4$ as above, we construct a vector bundle $E$ with $\mbox{det}(E)=\OO_C(1)$ and $h^0(C, E)=4$, lying in an exact sequence $$
0\longrightarrow E\longrightarrow W\otimes \OO_C(1)\longrightarrow \OO_C(2)\longrightarrow 0,$$
where $W\in G(3, H^0(C, \OO_C(1))$ has the property that the quadric $Q\in \mathrm{Sym}^2 H^0(C, \OO_C(1))$ induced by $S$ is representable by a tensor in $W\otimes H^0(C, L)$.  This construction is a particular procedure of associating vector bundles to non-trivial syzygies, cf. \cite{AN}. The proof that $E$ is stable is standard and proceeds along the lines of e.g. \cite{GMN} Theorem 3.2. Next we compute the Clifford invariant:
$$\gamma(E)=3a+\frac{b}{2}<ab-2=\mathrm{Cliff}(C),$$
since $b\geq 4$, so not only $\mathrm{Cliff}_2(C)<\mathrm{Cliff}(C)$, but the difference $\mathrm{Cliff}(C)-\mathrm{Cliff}_2(C)$ becomes arbitrarily positive.

\begin{proof}
By means of \cite{K} Theorem 6.1, there exist a smooth complete intersection surface $S\subset \PP^4$ of type $(2, 3)$ such that $\mbox{Pic} (S)=\mathbb Z\cdot H\oplus \mathbb Z\cdot C$, where $H^2=6,\ H\cdot C=d =6a+b$ and $C^2=2(g-1)$ (Note that such a surface exists when $d^2>12g$, which is satisfied when $b\geq 4$).  The divisor $E := C- aH$  verifies $E^2 = 0$, $E\cdot H = b$ and  $E\cdot C = ab$. In particular $E$ is effective. The class $E$ is primitive, hence it follows that $h^0(S, E)=h^0(C, \OO_C(E))=2$, where the last equality follows by noting that $H^1(S, \OO_S(E-C))=0$ by Kodaira vanishing. Furthermore, $h^1(C, \OO_C(E)) \geq  3a^2 +2$, that is,  $\OO_C(E)$ contributes to  $ \Cliff(C)$ and then we write that
$$
\gon (C) =\Cliff(C) + 2 \leq \Cliff(C, \OO_C(E)) +2 =ab.
$$
We shall show that $\OO_C(E) $ computes the Clifford index of $C$.
\\
First, we classify the primitive effective divisors $F\equiv mH+nC\subset S$ having self-intersection zero.  By solving the equation $(mH+nC)^2=0$, where $m, n\in \mathbb Z$, we find the following primitive solutions:  $E_1\equiv (3a+b)H - 3C$ \ for $b\neq 6$ (respectively $E_2\equiv (a+2)H-C$\ for $b=6$), and
$E_3=E\equiv C-aH$. A simple computation shows that $E_i\cdot C>ab$ for $i=1, 2$.

Since $\mathrm{Cliff}(C)\leq ab-2<[\frac{g-1}{2}]$, the Clifford index of $C$ is computed by a bundle defined on $S$. Following
\cite{GL2},  there exists an effective
divisor $D \equiv mH+nC$ on  $S$, satisfying the following numerical conditions:
\begin{equation}\label{con3}
h^0(S, D)=h^0(C, \OO_C(D))\geq 2, \ h^0(S, C-D)\geq 2,\ \
D^2\geq 0 \mbox{ and }\  D\cdot C\leq g-1,
\end{equation}
and such that
$$f(D):=\mathrm{Cliff}(\OO_C(D))+2=D\cdot C - D^2=\mathrm{Cliff}(C)+2.
$$
Furthermore, $D$ can be chosen such that $h^1(S, D)=0$, cf. \cite{M}. To bound $f(D)$ and show that $f(D)\geq ab$, we distinguish two cases depending on whether $D^2 > 0 $ or $D^2=0$.

By a complete classification of curves with self-intersection zero, we have already seen that for any elliptic pencil $|D|$ satisfying (\ref{con3}), one has $f(D)\geq ab=f(E)$. We are left with the case $D^2>0$ and rewrite the inequalities (\ref{con3}):
$$
\begin{array}{rl}
(i) & (6a+b)m + (2n-1) (3a^2+ab) \leq 0 \\
(ii) & (m+an)(3an+3m+bn) > 0\\
(iii) & 6m + (6a+b)n >2,
\end{array}
$$
where $(ii)$ comes from the assumption $D^2>0$ and $(iii)$ from the fact that $D\cdot H>2$. Furthermore,
\begin{equation} \label{quad}
f(m, n):= D\cdot C - D^2=  -6m^2 + m(d-2nd) + (n -n^2)(2g-2).
\end{equation}
We prove that for any divisor $D$ satisfying $(i)$ -$(iii)$, the inequality $f(m,n) \geq ab $ holds, from which we conclude that $\mathrm{Cliff}(C)=ab-2$.
\\
{\it Case $n <0$}. From  $(iii)$ we find that $m>0$. Then $m< -an$ or $3m> -(3a+b)n$.
When $m < -an$,  from $(iii)$ we have that
$
2 < 6m+dn < -6an +dn = nb <0,
$
which is a contradiction.  Suppose $(3a+b)n+3m>0$.  For a fixed $n$ the function $f(m,n)$ reaches its maximum
at $m_0:= \frac{d(1-2n)}{12}$. So when $3m_0+(3a+b)n\leq 0$, we have  $f(m, n) \geq f\bigl(\frac{(1-2n)(g-1)}{d}, n \bigr)$,
since by condition $(i)$,   $m \leq \frac{(1-2n)(g-1)}{d} $.
A simple computation gives that whenever $n<0$, one has the inequality:
\begin{eqnarray*}
f \left( \frac{(1-2n)(g-1)}{d}, n \right) &=&  (2n^2-2n)(g-1) \frac{b^2}{d^2} + (g-1)\left(1- \frac{6(g-1)}{d^2} \right)\\
&\geq& 4(g-1)\frac{b^2}{d^2} + \frac{g-1}{d^2}\bigl(18a^2+b^2+6ab\bigr)  \geq \frac{3a^2 +ab}{ 2} \geq ab.\\
\end{eqnarray*}

Assume now that $3m_0+(3a+b)n>0$. Since $ m \in \bigl(-\frac{(3a+b)n}{3}, \frac{(1-2n)(g-1)}{d} \bigr]$, we have
 $$
f(m,n) \geq \mbox{ min } \Bigl\{ f \Bigl(-\frac{(3a+b)n}{3}, n\Bigr), f\Bigl( \frac{(1-2n)(g-1)}{d} ,n\Bigr)\Bigr \}.
$$
A direct computation yields
$$
 f\Bigl(-\frac{(3a+b)n}{3}, n\Bigr)=-n  \left(ab + \frac{b^2}{3} \right)
\geq ab + \frac{b^2}{3} \geq ab.
$$
\\
\\
{\it Case $n >0$}. If $m \geq 0$ we get a contradiction to $(i)$. Suppose  $m<0$, then we have either $3m+(3a+b)n<0$, or else
$m> -an$. The first case contradicts  $(iii)$, so it does not appear. Suppose $m > -an$.  Reasoning as before, observe that  $m_0 < (1-2n)(g-1)/d$, where $m_0$ is the maximum of $f(m,n)$ for a fixed $n$,  and $m$ takes values in the interval $\bigl(-an, \frac{(1-2n)(g-1)}{d}\bigr ]$. If $-an\geq m_0$, then $f(m,n) \geq f\bigl(\frac{(1-2n)(g-1)}{d}, n\bigr)$.  Since we are assuming $-an < \frac{(1-2n)(g-1)}{d}$, we have that $n <  \frac{3a}{b} +1$.  We use this bound to directly show, like in the previous case, that $f\bigl(\frac{(1-2n)(g-1)}{d}, n\bigr)\geq ab$.
When $-an< m_0$ we have that
$$
f(m,n) \geq \Min \Bigl\{f(-an,n), f \Bigl(\frac{(1-2n)(g-1)}{d},n\Bigr) \Bigr\} .
$$
In this case it is enough to note that $f(-an,n)=nab \geq ab$.
\vskip 4pt
\noindent
{\it Case $n =0$}. From inequalities $(i)$ and $(iii)$ with $n=0$, we have $1\leq  m \leq \frac{g-1}{d}$.
Note that $f(m,0) = -6m^2 +md $ reaches its maximum at $\frac{d}{12}$.
So, since $\frac{g-1}{d} \leq \frac{d}{12} $, we conclude that $f(m,0)
\geq f(1,0)= 6a+b -6 $. Finally, we observe that $6a+b-6 \geq ab$ if and only if $b\leq 6$.
This finishes the proof.
\end{proof}

\section{The Fourier-Mukai involution on $\F_{11}$}
The aim of this section is to provide a detailed proof of Mercat's conjecture $(M_2)$ in one non-trivial case, that of genus $11$, and discuss the connection to Mukai's work \cite{M1}, \cite{M3}. We denote as usual by $\F_g$ the moduli space parametrizing pairs $[S, \ell]$, where $S$ is a smooth $K3$ surface and $\ell\in \mathrm{Pic}(S)$ is a primitive nef line bundle with $\ell^2=2g-2$. Furthermore, we introduce the parameter space
$$\P_g:=\bigl\{[S, C]: S \mbox{ is a smooth } K3 \mbox{ surface},\ \ \ C\subset S \mbox{ is a smooth curve}, \ \ [S, \OO_S(C)]\in \F_{g}\bigr\}$$
and denote by $\pi:\P_g\rightarrow \F_g$ the projection map $[S, C]\mapsto [S, \OO_S(C)]$.
If $S$ is a $K3$ surface, following \cite{M1}, we set $\widetilde{H}(S, \mathbb Z):=H^0(S, \mathbb Z)\oplus H^2(S, \mathbb Z)\oplus H^4(S, \mathbb Z)$
and $$\widetilde{NS}(S):=H^0(S, \mathbb Z)\oplus NS(S)\oplus H^4(S, \mathbb Z).$$  We recall the definition of the \emph{Mukai pairing} on $\widetilde{H}(S, \mathbb Z)$:
$$(\alpha_0, \alpha_2, \alpha_4)\cdot(\beta_0, \beta_2, \beta_4):=\alpha_2\cup \beta_2-\alpha_4\cup \beta_0-\alpha_0\cup \beta_4\in H^4(S, \mathbb Z)=\mathbb Z.$$
Let now $r, s\geq  1$ be relatively prime integers such that $g=1+rs$. For a polarized $K3$ surface $[S, \ell]\in \F_g$ one defines the \emph{Fourier-Mukai dual} $\hat{S}:=M_S(r, \ell, s)$, where
$$M_S(r, \ell, s)=\bigl\{E: E \mbox{ is an } \ell-\mbox{stable sheaf on } S, \ \mathrm{rk}(E)=r, \ c_1(E)=\ell,  \chi(S, E)=r+s\bigr\}.$$
Setting $v:=(r, \ell, s)\in \widetilde{H}(S, \mathbb Z)$, there is a Hodge isometry, see \cite{M1} Theorem 1.4:
$$\psi:H^2(M_S(r, \ell, s),  \mathbb Z)\stackrel{\cong}\longrightarrow v^{\perp}/\mathbb Zv.$$ We observe that $\hat{\ell}:=\psi^{-1}((0, \ell, 2s))$ is a nef primitive vector with $(\hat{\ell})^2=2g-2$, and in this way the pair $(\hat{S}, \hat{\ell})$ becomes a polarized $K3$ surface of genus $g$. The \emph{Fourier-Mukai involution} is the morphism $FM:\F_g\rightarrow \F_g$ defined by $FM([S, \ell]):=[\hat{S}, \hat{\ell}]$.

\vskip 3pt
We turn to the case $g=11$, when we set $r=2$ and $s=5$. For a general curve $[C]\in \cM_{11}$, the Lagrangian Brill-Noether locus
$$\mathcal{SU}_C(2, K_C, 7):=\{E\in \cU_C(2, 20): \mbox{det}(E)=K_C, \ h^0(C, E)=7\}$$
is a smooth $K3$ surface. The main result of \cite{M3} can be summarized as saying a general $[C]\in \cM_{11}$ lies on a unique $K3$ surface which moreover can be realized as $\widehat{\mathcal{SU}_C(2, K_C, 7)}$. Furthermore, there is a birational isomorphism
$$\phi_{11}:\cM_{11}\dashrightarrow \P_{11},\  \ \mbox{    } \phi_{11}([C]):=\bigl[\widehat{\mathcal{SU}_C(2, K_C, 7)}, C\bigr]$$
and we set $q_{11}:=\pi\circ \phi_{11}: \cM_{11}\dashrightarrow \F_{11}$.
On the moduli space $\cM_{11}$ there exist two distinct irreducible Brill-Noether divisors
$$\cM_{11, 6}^1:=\{[C]\in \cM_{11}: W^1_6(C)\neq \emptyset\} \ \mbox{ and } \cM_{11, 9}^2:=\{[C]\in \cM_{11}: W^2_9(C)\neq \emptyset\}.$$
Via the residuation morphism $W^1_6(C) \ni L\mapsto K_C\otimes L^{\vee}\in W^5_{14}(C)$,  the Hurwitz divisor is the pull-back of a Noether-Lefschetz divisor on $\F_{11}$, that is, $\cM_{11, 6}^1=q_{11}^*(D_6^1)$ where
$$D_6^1:=\{[S, \ell]\in \F_{11}: \exists H\in \mathrm{Pic}(S), \ H^2=8, \ H\cdot \ell=14\}.$$
Similarly, via the residuation map $W^2_9(C)\ni L\mapsto K_C\otimes L^{\vee}\in W^3_{11}(C)$, one has the equality of divisors $\cM_{11, 9}^2=q_{11}^*(D_9^2)$, where
$$D_9^2:=\{[S, \ell]\in \F_{11}: \exists H\in \mathrm{Pic}(S), \ H^2=4, \ H\cdot \ell=11\}.$$

Next we establish Mercat's conjecture for general curves of genus $11$.
\begin{theorem}\label{mercat11}
The equality $\mathrm{Cliff}_2(C)=\mathrm{Cliff}(C)$ holds for a general curve $[C]\in \cM_{11}$.
\end{theorem}
\begin{proof} We fix a curve $[C]\in \cM_{11}$ such that $(i)$ $W^1_7(C)$ is a smooth curve, $(ii)$ $W^2_9(C)=\emptyset$ (in particular, any Petri general curve will satisfy these conditions) and $(iii)$ the rank $2$ Brill-Noether locus $\mathcal{SU}_C(2, K_C, 7)$ is a smooth $K3$ surface of Picard number $1$. As discussed in both \cite{LMN} Proposition 4.5 and \cite{FO} Question 3.5, in order to verify $(M_2)$, it suffices to show that $C$ possesses no bundles $E\in \cU_C(2, 13)$ with $h^0(C, E)=4$. Suppose $E$ is such a vector bundle. Then $L:=\mbox{det}(E)\in W^4_{13}(C)$ is a linear series such that the multiplication map $\nu_2(L):\mbox{Sym}^2 H^0(C, L)\rightarrow H^0(C, L^{\otimes 2})$ is not injective. For each extension class
$$e\in \PP_L:=\PP(\mathrm{Coker}\ \nu_2(L))^{\vee}\subset \PP(H^0(C, L^{\otimes 2}))^{\vee}=\PP \mathrm{Ext}^1(L, K_C\otimes L^{\vee}),$$
one obtains a rank $2$ vector bundle $F$ on $C$ sitting in an exact sequence
\begin{equation}\label{ext11}
0\longrightarrow K_C\otimes L^{\vee}\longrightarrow F\longrightarrow L\longrightarrow 0,
\end{equation}
such that $h^0(C, F)=h^0(C, L)+h^0(C, K_C\otimes L^{\vee})=7$. We claim that any non-split vector bundle $F$ with $h^0(C, F)=7$ and which sits in an exact sequence (\ref{ext11}), is semistable. Indeed, let us assume by contradiction that $M\subset F$ is a destabilizing line subbundle with $\mbox{deg}(M)\geq 11$. Since $\mbox{deg}(M)>\mbox{deg}(K_C\otimes L^{\vee})$, the composite morphism $M\rightarrow L$ is non-zero, hence we can write that $M=L(-D)$, where $D$ is an effective divisor of degree $1$ or $2$. Because $W^2_9(C)=\emptyset$, one finds that $h^0(C, K_C\otimes L^{\vee}(D))=2$ and $L$ must be very ample, that is, $h^0(C, L(-D))=h^0(C, L)-\mbox{deg}(D)$. We obtain that $$
h^0(L)+h^0(K_C\otimes L^{\vee})=h^0(F)\leq h^0(M)+h^0(K_C\otimes M^{\vee})=h^0(L)-\mbox{deg}(D)+h^0(K_C\otimes L^{\vee}),$$
a contradiction. Thus one obtains an induced morphism $u:\PP_L \rightarrow \mathcal{SU}_C(2, K_C, 7)$. Since $\mathcal{SU}_C(2, K_C, 7)$ is a $K3$ surface, this also implies that $\mathrm{Coker} \ \nu_2(L)$ is $2$-dimensional, hence $\PP_L=\PP^1$.

We claim that $u$ is an embedding. Setting $A:=K_C\otimes L^{\vee}\in W^1_7(C)$, we write the exact sequence  $0\rightarrow H^0(C, \OO_C)\rightarrow H^0(C, F^{\vee}\otimes L)\rightarrow H^0(C, K_C\otimes A^{\otimes (-2)})$, and note that the last vector space is the kernel of the Petri map $H^0(C, A)\otimes H^0(C, L)\rightarrow H^0(C, K_C)$, which is injective, hence $h^0(C, F^{\vee}\otimes L)=1$. This implies that $u$ is an embedding. But this contradicts the fact that $\mathrm{Pic}\ \mathcal{SU}_C(2, K_C, 7)=\mathbb Z$, in particular $\mathcal{SU}_C(2, K_C, 7)$ contains no $(-2)$-curves. We conclude that $\nu_2(L)$ is injective for \emph{every} $L\in W^4_{13}(C)$.
\end{proof}
This proof also shows that the failure locus of statement $(M_2)$ on $\cM_{11}$ is equal to the Koszul divisor
$$\mathfrak{Syz}_{11, 13}^4:=\{[C]\in \cM_{11}: \exists L\in W^4_{13}(C) \mbox{ such that } K_{1, 1}(C, L)\neq 0\}.$$
Suppose now that $[C]\in \mathfrak{Syz}_{11, 13}^4$ is a general point corresponding to an embedding $C\stackrel{|L|}\hookrightarrow \PP^4$ such that $C$ lies on a $(2, 3)$ complete intersection $K3$ surface $S\subset \PP^4$. Then $S=\widehat{\mathcal{SU}_C(2, K_C, 7)}$ and $\rho(S)=2$ and furthermore
$\mathrm{Pic}(S)=\mathbb Z\cdot C\oplus \mathbb Z\cdot H$, where $H^2=6, C\cdot H=13$ and $C^2=20$. In particular we note that $S$ contains no $(-2)$-curves, hence $S$ and $\hat{S}$ are not isomorphic.

Let us define the Noether-Lefschetz divisor $$D_{13}^4:=\{[S, \ell]\in \F_{11}: \exists H\in \mathrm{Pic}(S), \ H^2=6, H\cdot \ell=13\},$$ therefore
$\mathfrak{Syz}_{11, 13}^4=q_{11}^*(D_{13}^4)$.

\begin{proposition}
The action of the Fourier-Mukai involution $FM:\F_{11}\rightarrow \F_{11}$ on the three distinguished Noether-Lefschetz divisors is described as follows:
\begin{enumerate}
\item $FM(D_6^1)=D_6^1$.
\item $FM(D_9^2)=\{[S, \ell]\in \F_{11}: \exists R\in \mathrm{Pic}(S)  \ \mbox{ such that } R^2=-2, \ R\cdot \ell=1\}.$
\item $FM(D_{13}^4)=\{[S, \ell]\in \F_{11}: \exists R\in \mathrm{Pic}(S) \ \mbox{such that } R^2=-2, \ R\cdot \ell=3\}.$
\end{enumerate}
\end{proposition}
\begin{proof} For $[S, \ell]\in \F_{11}$, we set $v:=(2, \ell, 5)\in \widetilde{H}(S, \mathbb Z)$ and $\hat{\ell}:=(0, \ell, 10)\in \widetilde{H}(S, \mathbb Z)$ for the class giving the genus $11$ polarization. We describe the
lattice $\psi\bigl(NS(\hat{S})\bigr)\subset \widetilde{NS}(S)$.

In the case of a general point of $D_6^1$ with lattice $NS(S)=\mathbb Z\cdot \ell\oplus \mathbb Z\cdot H$, by direct calculation we find that
 $\psi\bigl(NS(\hat{S})\bigr)$ is generated by the vectors $\hat{\ell}$ and $(2, \ell+H, 12)$. Furthermore, $(2, \ell+H, 12)^2=8$ and $(2, H+\ell, 12)\cdot \hat{\ell}=14$, that is, $\mbox{Pic}(\hat{S})\cong \mbox{Pic}(S)$, hence $D_6^1$ is a fixed divisor for the automorphism $FM$.
 \vskip 3pt

 A similar reasoning for a general point of the divisor $D_9^2$ shows that the Neron-Severi groups $\psi\bigl(NS(\hat{S})\bigr)$ is generated by $\hat{\ell}$  and $(-1, H-\ell, -2)$, where $(-1, H-\ell, -2)^2=-2$ and $(-1, H-\ell, -2)\cdot \hat{\ell}=1$. In other words, the class $(-1, H-\ell, -2)$ corresponds to a line in the embedding $\hat{S}\stackrel{|\hat{\ell}|}\hookrightarrow \PP^{11}$. Finally, for a general point of $D_{13}^4$ corresponding to a lattice $\mathbb Z\cdot \ell\oplus \mathbb Z\cdot H$, the Picard lattice of the Fourier-Mukai partner is spanned by the vectors $\hat{\ell}$ and $(-1, H-\ell, -1)$, where $(-1, H-\ell, -1)^2=-2$ and $(-1, H-\ell, -1)\cdot \hat{\ell}=3$.
\end{proof}
\begin{remark} The fact that the divisor $D_6^1$ is fixed by the automorphism $FM$ is already observed and proved with geometric methods in \cite{M3}
Theorem 3.
\end{remark}
\begin{remark} It is instructive to point out the difference between a general element of $D^4_{13}$ and its Fourier-Mukai partner. As a polarized $K3$ surface,  $\mathcal{SU}_C(2, K_C, 7)$ is characterized by the existence of a degree $3$ rational curve $u(\PP_L)\subset \mathcal{SU}_C(2, K_C, 7)$. On the other hand, the complete intersection surface $S\subset \PP^4$ containing $C\stackrel{|L|}\hookrightarrow \PP^4$, where $L\in W^4_{13}(C)$, carries no smooth rational curves. It contains however elliptic curves in the linear system $|\OO_S(C-H)|$. Thus the involution $FM$ assigns to a $K3$ surface with a degree $7$ elliptic pencil, a $K3$ surface containing a $(-2)$-curve. Since $S=\widehat{\mathcal{SU}_C(2, K_C, 7)}$, it also follows that the complete intersection $S$ is a smooth $K3$ surface, which a priori is not at all obvious.
\end{remark}


\begin{thebibliography}{EMS}
\bibitem[AN]{AN} M. Aprodu and J. Nagel, {\em{Non-vanishing for Koszul cohomology of curves}}, Comment. Math. Helvetici \textbf{82} (2007), 617-628.
\bibitem[CLM]{CLM} C. Ciliberto, A. Lopez and R. Miranda, {\em{Projective degenerations of $K3$ surfaces, Gaussian maps and Fano threefolds}}, Inventiones Math. \textbf{114} (1993), 641-667.
\bibitem[DM]{DM} R. Donagi and D. Morrison,  {\em {Linear systems on K3 sections}}, Journal of Differential Geometry \textbf{29} (1989), 49-64.
\bibitem[F]{F1} G. Farkas, {\em{Brill-Noether loci and the gonality stratification of $\cM_g$}}, J. reine angew. Mathematik \textbf{539} (2001), 185-200.
\bibitem[FO]{FO} G. Farkas and A. Ortega, {\em{The maximal rank conjecture and rank two Brill-Noether theory}}, Pure Appl. Math. Quarterly \textbf{7} (2011), 1265-1296.
\bibitem[G]{G} M. Green, {\em{Koszul cohomology and the cohomology of projective varieties}}, Journal of Differential Geometry \textbf{19} (1984), 125-167.
\bibitem[GL1]{GL} M. Green and R. Lazarsfeld, {\em{The non-vanishing of certain Koszul cohomology groups}}, Journal of Differential Geometry \textbf{19}
(1984), 168-170.
\bibitem[GL2]{GL2} M. Green and R. Lazarsfeld, {\em{Special divisors on curves on a $K3$ surface}}, Inventiones Math. \textbf{89} (1987), 357-370.
\bibitem[GMN]{GMN} I. Grzegorczyk, V. Mercat and P.E. Newstead, {\em{Stable bundles of rank $2$ with $4$ sections}}, International Journal of Mathematics \textbf{22} (2011), 1743-1762.
\bibitem[HL]{HL} D. Huybrechts and M. Lehn, {\em{The geometry of moduli spaces of sheaves}}, Second Edition, Cambridge University Press 2010.
\bibitem[K]{K} A. Knutsen, {\em{Smooth curves on projective $K3$ surfaces}}, Math. Scandinavica \textbf{90} (2002), 215-231.
\bibitem[LN1]{LN1} H. Lange and P.E. Newstead, {\em{Clifford indices for vector bundles on curves}}, arXiv:0811.4680, in: Affine Flag Manifolds and Principal Bundles (A. H. W. Schmitt editor), Trends in Mathematics, 165-202, Birkh\"auser (2010).
\bibitem[LN2]{LN2} H. Lange and P.E. Newstead, {\em{Further examples of stable bundles of rank $2$ with $4$ sections}}, Pure Appl. Math. Quarterly \textbf{7} (2011), 1517-1528.
\bibitem[LN3]{LN3} H. Lange and P.E. Newstead, {\em{Vector bundles of rank $2$ computing Clifford indices}}, arXiv:1012.0469.
\bibitem[LMN]{LMN} H. Lange, V. Mercat and P.E. Newstead, {\em{On an example of Mukai}}, arXiv:1003.4007, to appear in Glasgow Math. Journal.
\bibitem[L]{L} R. Lazarsfeld, {\em{Brill-Noether-Petri without degenerations}}, Journal of Differential Geometry \textbf{23} (1986), 299-307.
\bibitem[Ma]{M} G. Martens,  {\em{On curves on $K3$ surfaces}}, in: Algebraic curves and projective geometry, Trento 1998,  Lecture Notes in Mathematics Vol. 1389, Springer (1989), 174-182.
\bibitem[Me]{Me} V. Mercat, {\em{Clifford's theorem and higher rank vector bundles}}, International Journal of Mathematics \textbf{13} (2002), 785-796.
\bibitem[M1]{M1} S. Mukai, {\em{On the moduli space of bundles on $K3$ surfaces}}, in: Vector bundles on algebraic varieties, Tata Institute of Fundamental Research Studies in Mathematics, vol. 11 (1987), 341-413.
\bibitem[M2]{M2} S. Mukai, {\em{Biregular classification of Fano $3$-folds and Fano manifolds of coindex $3$}}, Proceedings of the  National Academy Sciences USA,
Vol. 86, 3000-3002, 1989.
\bibitem[M3]{M3} S. Mukai, {\em{Curves and $K3$ surfaces of genus eleven}}, in: Moduli of vector bundles, Lecture Notes in Pure and Applied Math. Vol. 179,  Dekker (1996), 189-197.
\bibitem[M4]{M4} S. Mukai, {\em{Curves and symmetric spaces II}}, Annals of Mathematics \textbf{172} (2010),  1539-1558.
\bibitem[PR]{PR} K. Paranjape and S. Ramanan, {\em{On the canonical ring of a curve} }, in: Algebraic Geometry and commutative Algebra (1987), 503-516.
\bibitem[V1]{V1} C. Voisin, {\em{Sur l'application de Wahl des courbes satisfaisant la condition de Brill-Noether-Petri}}, Acta Math. \textbf{168} (1992), 249-272.
\bibitem[V2]{V2}
C. Voisin, {\em{Green's canonical syzygy conjecture for generic curves of odd genus}},
Compositio Math. \textbf{141} (2005), 1163-1190.

\end{thebibliography}
 \end{document}